\documentclass[11pt]{amsart}
\usepackage{extarrows}
\usepackage[colorlinks, citecolor=blue, dvipdfm, pagebackref]{hyperref}

\newtheorem{theorem}{Theorem}[section]
\newtheorem{lemma}[theorem]{Lemma}

\theoremstyle{definition}

\newtheorem{proposition}[theorem]{Proposition}
\newtheorem{corollary}[theorem]{Corollary}
\newtheorem{remark}[theorem]{Remark}

\theoremstyle{remark}

\newcommand{\be}{\begin{equation}}
\newcommand{\ee}{\end{equation}}

\numberwithin{equation}{section}

\begin{document}
\title{The Alexandrov-Fenchel type inequalities, revisited}
\author{Ping Li}
\address{School of Mathematical Sciences, Fudan University, Shanghai 200433, China}
\email{pingli@tongji.edu.cn\\
pinglimath@gmail.com}
\thanks{The author was partially supported by the National
Natural Science Foundation of China (Grant No. 11722109).}

 \subjclass[2010]{32Q15, 52A39, 15A45.}


\keywords{Alexandrov-Fenchel type inequality, Brunn-Minkowski inequality, mixed volume, mixed discriminant, Hodge index theorem, mixed Hodge index theorem, mixed Hodge-Riemann bilinear relation, K\"{a}hler class, nef class, big class.}

\begin{abstract}
Various Alexandrov-Fenchel type inequalities have appeared and played important roles in convex geometry, matrix theory and complex algebraic geometry. It has been noticed for some time that they share some striking analogies and have intimate relationships. The purpose of this article is to shed new light on this by comparatively investigating them in several aspects. \emph{The principal result} in this article is a complete solution to the equality characterization problem of various Alexandrov-Fenchel type inequalities for intersection numbers of nef and big classes on compact K\"{a}hler manifolds, extending some earlier related results. In addition to this central result, we also give a geometric proof of the complex version of the Alexandrov-Fenchel inequality for mixed discriminants and a determinantal generalization of various Alexandrov-Fenchel type inequalities.
\end{abstract}

\maketitle

\section{Introduction}\label{introduction}
One of the most fundamental results in convex geometry is the Alexandrov-Fenchel (AF for short) inequality for mixed volumes on convex bodies in the Euclidean spaces. There are at least three different proofs to this classical result. The original one is due to Alexandrov and Fenchel independently around 1936 (\cite{Al1}, \cite{Fe}). Soon afterwards Alexandrov introduced the notion of mixed discriminants for matrices and took it up as a tool to derive his second proof of this inequality (\cite{Al2}). Around the year 1979 Khovanskii and Teissier discovered independently a profound link between the theory of mixed volumes and algebraic geometry (\cite{Kh}, \cite{Te1}), which leads to a third proof of the AF inequality using the Hodge index theorem in algebraic geometry (cf. \cite[\S 27]{BZ}). Since then, many efforts are devoted to exploring deeper relationships among them and have produced fruitful results. In particular, along this line several kinds of AF type inequalities and related results were discovered. We refer to \cite{Te2}, \cite{Te3}, \cite{Gr}, \cite{Ti}, \cite{BFJ}, \cite{LX1}, \cite{LX2}, \cite{Xi}, \cite{DX}, \cite{RT}, and the related references therein. Some log-concave type sequences naturally arise from AF type inequalities and in K\"{a}hler geometry the latter is a direct consequence of the mixed Hodge-Riemann relation (\cite{DN}, \cite{DN2}, \cite{Ca}). This inspired Huh and his coauthors to develop Hodge theory for purely combinatorial objects and resolve several long-standing conjectures related to log-concavity in them. We refer the reader to Huh's ICM Lecture \cite{Huh} for a survey of these remarkable results.

\emph{The main purpose} of this article is to comparatively investigate three kinds of AF type inequalities in several aspects. We have three main results as well as some consequences and in what follows we shall briefly describe them.

\emph{The first result} is to apply an AF type inequality on compact K\"{a}hler manifolds established by the author in \cite{Li2} to obtain a complex Hermitian version of the AF type inequality for mixed discriminants, Theorem \ref{AF complex mixed discriminant}. The reason for this is two-fold. On the one hand, we apply an AF type inequality of algebro-geometric nature to yield a different proof of that of a purely combinatorial nature, revealing the intimate relationships between them and thus fitting into the theme of our article very well. On the other hand and more importantly, although the real version of this inequality, which is due to Alexandrov (\cite{Al2}), has been well-known for a long time and have several different proofs up to now (\cite[\S 5.5]{Sc}, \cite[\S 3]{Le}), it seems that the complex version of this inequality and/or its detailed proof never appeared in existing literature, at least to the author's best knowledge. Even if it should be known to be true to some experts in inequalities for matrices, for instance, Alexandrov himself gave a very short remark about the validity of this inequality for complex Hermitian matrices at the bottom of second page in his original paper \cite{Al2}, which was pointed out to the author by R.B. Bapat, it deserves to be circulated by presenting an explicit and detailed proof.

In \cite{Sh} Shephard generalized the AF inequality for mixed volumes to a determinantal case. Careful investigation shall find that Shephard's application of the AF inequality for mixed volumes is \emph{formal} and indeed is valid for abstract bilinear functions satisfying such inequalities. Inspired by this observation, our \emph{second result}, Theorem \ref{determinantal}, gives a determinantal type inequality for general bilinear functions satisfying AF type inequalities on closed cones of real Euclidean spaces. Applying this abstract result to mixed discriminants on matrices and K\"{a}hler/nef classes on compact K\"{a}hler manifolds yields respectively Corollaries \ref{coro1} and \ref{coro2}, which similarly generalize the original AF type inequalities to determinantal cases. Theorem \ref{determinantal} and Corollaries \ref{coro1} and \ref{coro2} depend on a positive integer $r$. The case of $r=1$ reduces to the original AF type inequality and \emph{the case of $r=2$ shall play a key role} in our proof of the central result Theorem \ref{equality case1}.

One unsolved problem involved in the AF inequality for mixed volumes is to completely characterize its equality case (\cite[\S 7.6]{Sc}). A similar characterization problem can also be posed to the equality case of the AF type inequality of the intersection numbers of nef classes on projective or compact K\"{a}hler manifolds. This problem was first proposed and studied by Teissier for the Khovanskii-Teissier inequalities where only two nef classes are involved (\cite{Te3}) and so in some literature this problem is referred to as \emph{Teissier's problem}. Teissier's problem was recently solved for a pair of nef and big classes by Bouchsom, Favre and Jonsson in the context of projective manifolds (\cite{BFJ}) and by Fu and Xiao for general compact K\"{a}hler manifolds (\cite{FX2}). Another extremal case was also recently solved by Lehmann and Xiao (\cite{LX1}). \emph{The third result}, which is also \emph{our central result} in this article, Theorems \ref{equality case1} and \ref{equality case2}, completely settle the characterization problem of the equality cases of the general AF type inequalities for nef and big classes on compact K\"{a}hler manifolds.

\emph{The rest of this article is organized as follows}. In Section \ref{Preliminaries and background materials} we shall recall three AF type inequalities for mixed volumes, mixed discriminants and intersection numbers of K\"{a}hler/nef classes on compact K\"{a}hler manifolds respectively, as well as set up some notation and symbols for our later purpose. We will state our three main results, Theorems \ref{AF complex mixed discriminant}, \ref{determinantal} and \ref{equality case1}, as well as their consequences in Section \ref{main results}. Then Sections \ref{Proof of AF complex mixed discriminant}, \ref{Proof of determinantal} and \ref{Proof of equality} are devoted to respectively the proofs of these three main results. It turns out that various original AF type inequalities can be viewed as cases of level two and can be repeatedly applied via a unified induction argument to be extended to the cases of any level \big((\ref{AF mixed volume})$\Rightarrow$(\ref{AF mixed volume for m}), (\ref{AF kahler class})$\Rightarrow$(\ref{AF kahler class for m}), (\ref{AF complex mixed discriminant})$\Rightarrow$(\ref{AF complex mixed formula for m}),  Theorem \ref{equality case1}$\Rightarrow$Theorem \ref{equality case2}\big). This induction argument should be classical and well-known to related experts, but we cannot find a very detailed and clean argument in the literature and thus, for the reader's convenience, the last section, Section \ref{last section} entitled with ``Appendix'', is included to illustrate this induction argument.

\section*{Acknowledgements}
The author would like to express his sincere thanks to Professors R.B. Bapat, L. Gurvits, and R. Schneider for their useful comments on the materials of mixed discriminants, and Dr. Jian Xiao for his many useful comments on the materials related to Theorems \ref{equality case1} and \ref{equality case2}. Special thanks go to Professor R.B. Bapat who pointed out to the author the short remark on the validity of the complex version of the AF type inequality for mixed discriminants given by Alexandrov in \cite{Al2}.

\section{Preliminaries and background materials}\label{Preliminaries and background materials}
We review in this section the AF type inequalities respectively for mixed volumes, mixed discriminants and intersection numbers of K\"{a}hler/nef classes on compact K\"{a}hler manifolds, and along this line set up some necessary notation and symbols used in later sections.

Let $\mathcal{K}^n$ be the set consisting of all non-empty compact convex subsets in $\mathbb{R}^n$. For each $K\in\mathcal{K}^n$, denote by $\text{V}(K)$ its $n$-dimensional volume in $\mathbb{R}^n$, which is positive when $K$ has a non-empty interior. For  $K_1,\ldots,K_n\in\mathcal{K}^n$, their \emph{mixed volume}, denoted by $V(K_1,\ldots,K_n)$, can be directly defined via the polarization formula:
$$V(K_1,\ldots,K_n):=\frac{1}{n!}\sum_{(\epsilon_1,\ldots,\epsilon_n)
\in\{0,1\}^n}(-1)^{n+\sum_{i=1}^n\epsilon_i}
\cdot V(\sum_{i=1}^n\epsilon_iK_i),$$
from which it is clear that the mixed volume $V(\cdot,\ldots,\cdot)$ is symmetric in its arguments.

It turns out that the mixed volume is nonnegative and satisfies the following well-known fact due to Minkowski:
\be\label{expansionmixedvolume}V(\sum_{i=1}^m\lambda_iK_i)=
\sum_{\mbox{\tiny$\begin{array}{c}
i_1+\cdots+i_m=n\\
0\le i_1,\ldots,i_m \le n\end{array}$}}\frac{n!}{i_1!\cdots i_m!}V(K_{1}[i_1],\ldots,K_{m}[i_m])
\lambda_1^{i_1}\cdots\lambda_m^{i_m},\qquad(\lambda_i\geq0),\ee
where the compact convex subset $\sum_{i=1}^m\lambda_iK_i$ stands for the Minkowski sum $$\sum_{i=1}^m\lambda_iK_i:=\{\sum_{i=1}^m\lambda_ik_i~|~k_i\in K_i\},$$
 $m$ may be different from $n$ in general, and the following notation is adopted, which shall be frequently used in the sequel:
\be\label{simplenotation}V(K_{1}[i_1],\ldots,K_{m}[i_m])
:=V(\underbrace{K_1,\ldots,K_1}_{i_1},\ldots,\underbrace{K_m,
\ldots,K_m}_{i_m}).\ee
In other words, $n!\cdot V(K_1,\ldots,K_n)$ is the coefficient in front of $\lambda_1\cdots \lambda_n$ in the polynomial $V(\sum_{i=1}^n\lambda_iK_i)$ of $\lambda_1,\ldots,\lambda_n$.
More basic properties and discussions related to mixed volumes can be found in \cite[\S5]{Sc}.

One of the most fundamental results in convex geometry is the following Alexandrov-Fenchel inequality for mixed volumes, which was discovered independently by Alexandrov and Fenchel around in 1936 (\cite{Al1}, \cite{Fe}, cf. \cite[\S 7.3]{Sc}):
\begin{theorem}[AF inequality for mixed volumes]
Suppose that $K_1,\ldots,K_n\in\mathcal{K}^n$. Then we have
\be\label{AF mixed volume}V(K_1,K_2,K_3,\ldots,K_n)^2\geq V(K_1,K_1,K_3,\ldots,K_n)\cdot V(K_2,K_2,K_3,\ldots,K_n).\ee
\end{theorem}

\begin{remark}
\begin{enumerate}
\item
We refer the interested reader to \cite[p. 398]{Sc} about the historical remarks on the proof of this inequality.

\item
The equality case in (\ref{AF mixed volume}) clearly holds if $K_1$ and $K_2$ are homothetic: $K_1=\lambda K_2+t$ with $t\in\mathbb{R}^n$ and $\lambda>0$.
However, this is \emph{not} the only possibility for equality and the complete characterization of the equality cases is an unsolved problem and only partial results are known (cf. \cite[\S 7.6]{Sc}).
\end{enumerate}
\end{remark}

Many important results, including those series of geometric inequalities of isoperimetric type, turn out to be special cases of (\ref{AF mixed volume}) (\cite[\S 20.2]{BZ}). Here we only mention two direct improvement/consequence for our later purpose.

Since the mixed volume $V(\cdot,\ldots,\cdot)$ is symmetric in its arguments, repeated applications of (\ref{AF mixed volume}) yield a more general inequality (\cite[\S 7.4]{Sc})
\be\label{AF mixed volume for m}
V(K_1,\ldots,K_n)^m\geq\prod_{i=1}^m V(K_i[m],K_{m+1}\ldots,K_n),\qquad\forall~2\leq m\leq n,
\ee
and the idea of this proof can be found in Section \ref{last section}.

Assume $2\leq m\leq n$ and $K_0$, $K_1,$ $K_{m+1},$ $\ldots$, $K_n\in\mathcal{K}^n$, and define
\be\label{f(lambda)}f(\lambda):=V\Big(\big((1-\lambda)K_0+\lambda K_1\big)[m],K_{m+1},\ldots,K_n\Big)^{\frac{1}{m}},\qquad0\leq\lambda\leq 1.\ee

Directly showing that $f''(\lambda)\leq 0$ via (\ref{AF mixed volume}), another direct consequence is the following general Brunn-Minkowski (BM for short) theorem (\cite[p. 406]{Sc}).
\begin{corollary}[General BM theorem for mixed volumes]\label{AF mixed volume for BM}
The function $f(\lambda)$ defined by  (\ref{f(lambda)}) is concave on $[0,1]:$
\be\label{Brunn-Minkowski}
\begin{split}
&V\Big(\big((1-\lambda)K_0+\lambda K_1\big)[m],K_{m+1},\ldots,K_n\Big)^{\frac{1}{m}}\\
 \geq&(1-\lambda)V(K_{0}[m],K_{m+1},\ldots,K_n)^{\frac{1}{m}} +\lambda V(K_{1}[m],K_{m+1},\ldots,K_n)^{\frac{1}{m}},\qquad 0\leq\lambda\leq 1.\end{split}\ee
Furthermore, if the equality case in (\ref{Brunn-Minkowski}) holds for \emph{some} $\lambda\in(0,1)$, it must hold for \emph{all} $\lambda\in[0,1]$:
$$f(\lambda)\equiv(1-\lambda)f(0)+\lambda f(1),\qquad\forall~\lambda\in[0,1].$$
\end{corollary}

We now turn to the notion and inequalities of mixed discriminants for matrices. Mixed discriminants were introduced and investigated by Alexandrov as a tool to derive his second proof of the AF inequality (\ref{AF mixed volume}). Let $$A_r=\big(a^{(r)}_{ij}\big)^n_{i,j=1},\qquad r=1,\ldots,n,$$ be $n$ real or complex valued $n\times n$ matrices, which are parametrized by $r$ and, for each $r$, whose entries are parametrized by $i$ and $j$. The \emph{mixed discriminant} of $A_1,\ldots,A_n$, denoted by $D(A_1,\ldots,A_n)$, is defined to be the polarization of the determinant function:
$$D(A_1,\ldots,A_n):=\frac{1}{n!}\sum_{\sigma\in S_n}\begin{vmatrix}
a^{(\sigma(1))}_{11}&\cdots&a^{(\sigma(n))}_{1n}\\
\vdots&\ddots&\vdots\\
a^{(\sigma(1))}_{n1}&\cdots&a^{(\sigma(n))}_{nn}
\end{vmatrix},$$
where $S_n$ is the group of all permutations of the set $\{1,\ldots,n\}$.

$D(\cdot,\ldots,\cdot)$ is symmetric in its arguments, $D(A,\ldots,A)=\det(A)$, and satisfies
$$\det(\sum_{r=1}^m\lambda_iA_r)=\sum_{\mbox{\tiny$\begin{array}{c}
r_1+\cdots+r_m=n\\
0\le r_1,\ldots,r_m \le n\end{array}$}}\frac{n!}{r_1!\cdots r_m!}D(A_1[r_1],\ldots,A_m[r_m])\lambda_1^{r_1}\cdots\lambda_m^{r_m},$$
an analogy to the identity (\ref{expansionmixedvolume}), i.e., $n!D(A_1,\ldots,A_n)$ is the coefficient in front of $\lambda_1\cdots\lambda_n$ in the polynomial $\det(\sum_{i=1}^n\lambda_iA_i)$ of $\lambda_1,\ldots,\lambda_n$. Here we adopt a similar notation introduced in (\ref{simplenotation}).

If $A$ is a real symmetric positive definite (resp. positive semi-definite) matrix, we write $A>0$ (resp. $A\geq0$). The following AF type inequality was proved by Alexandrov in 1938 (\cite{Al2}, \cite[p. 1062]{Le}, \cite[p. 327]{Sc}), which turns out to be a special case of a general AF inequality for hyperbolic polynomials established by G{\aa}rding (\cite{Ga}, \cite[\S5.5]{Sc}).

\begin{theorem}[AF inequality for mixed discriminants]
Let $A$, $B$, $A_3,\ldots,A_n$ be $n$ real symmetric $n\times n$ matrices, where $A, A_3,\ldots,A_n>0$ and $B$ is arbitrary. Then
\be\label{AF real mixed discriminant}
D(A,B,A_3,\ldots,A_n)^2\geq D(A,A,A_3,\ldots,A_n)\cdot D(B,B,A_3,\ldots,A_n),\ee
with equality if and only if $A$ and $B$ are proportional: $B=\lambda A$ for some real number $\lambda$.
\end{theorem}

\begin{remark}
\begin{enumerate}
\item
Except the application by Alexandrov himself to yield a second proof of his inequality (\ref{AF mixed volume}), inequality (\ref{AF real mixed discriminant}) was overlooked for a long time until Egorychev used it to give a proof of the Van der Waerden's conjecture on the minimum of the permanent of a doubly stochastic matrices (\cite{Eg}). Since then, new interests in the inequalities related to the mixed discriminants arose and we refer the reader to \cite{Pa}, \cite{Ba} and \cite{Gu} and the references therein.

\item
By continuity, inequality (\ref{AF real mixed discriminant}) remains true if we only assume that the matrices $A,$ $A_3,$ $\ldots,$ $A_n$ are nonnegative definite. However, unlike positive definiteness case, in this situation it is difficult to characterize the equality case. Moreover, similar to the direct applications of (\ref{AF mixed volume}) to yielding (\ref{AF mixed volume for m}) and (\ref{Brunn-Minkowski}), we also have similar such results for mixed discriminants.
\end{enumerate}
\end{remark}

Around in 1979 Khovanskii and Teissier (\cite{Kh}, \cite{Te1}) independently discovered AF type inequalities for intersection numbers of nef divisors on projective manifolds based on the usual Hodge index theorem, which leads to an algebraic proof of (\ref{AF mixed volume}) (cf. \cite[\S 27]{BZ}) and thus establishes an intimate relationship between the theory of mixed volumes and algebraic geometry. These inequalities were reproved along similar lines by Beltrametti-Biancofiore-Sommese, who applied them to the study of projective manifolds of log-general type (\cite[p. 832]{BBS}, also cf. \cite[\S 1.6]{La}). Later Demailly  extended these results to compact K\"{a}hler manifolds (\cite[\S 5]{De}). The proof of Khovanskii-Teissier inequalities
is to apply the usual Hodge index theorem to the ample divisors, together with induction and continuity arguments. The approach also suggests that the usual
Hodge index theorem may be extended to the mixed case.
After some partial results towards this direction
(\cite{Gr}, \cite{Ti}), the mixed version of the Hodge index theorem was established in its full generality
by Dinh and Nguy$\hat{\text{e}}$n in \cite{DN}. Using this mixed Hodge index theorem, the author obtained in \cite{Li2} some Cauchy-Schwarz type inequalities for higher-dimensional cohomology classes on compact K\"{a}hler manifolds, both extended Khovanskii-Teissier inequalities and his previous work \cite{Li1}.

Before stating the next AF type inequality, we recall and set up some notation. Assume that $M$ is a compact connected K\"{a}hler manifold of complex dimension $n$. Let $\mathcal{K}\subset H^{1,1}(M;\mathbb{R})$ be the K\"{a}hler cone of
$M$, which is an open cone consisting of all K\"{a}hler classes on $M$. A $(1,1)$ cohomology class $\gamma\in H^{1,1}(M,\mathbb{R})$ is called a \emph{nef} class if
$\gamma\in\overline{\mathcal{K}}$, the closure of the K\"{a}hler cone. So any nef
class can be approximated by K\"{a}hler classes. For $\alpha_1,\ldots,\alpha_n\in H^{1,1}(M,\mathbb{R})$, denote by $\alpha_1\cdot\alpha_2\cdots\alpha_n$ their intersection number:
$$\alpha_1\cdot\alpha_2\cdots\alpha_n:=
\int_M\alpha_1\wedge\cdots\wedge\alpha_n\in\mathbb{R}.$$

The following AF type inequality (\ref{AF kahler class}) is a special case of the main results in \cite{Li2} (see \cite[Coro. 1.4]{Li2}), whose degenerated case (\ref{AF nef class}) was due to Demailly-Peternell (\cite[Prop. 2.5]{DP}, also cf. \cite[Coro. 3.1]{Li2}).
\begin{proposition}[AF type inequality for K\"{a}hler/nef classes]\label{AF kahler/nef}
Assume that $\alpha\in H^{1,1}(M,\mathbb{R})$ is an \emph{arbitrary} real $(1,1)$ cohomology class, $c,c_3,\ldots,c_n\in\mathcal{K}$ are K\"{a}hler classes, and $\gamma,\gamma_3,\ldots,\gamma_n\in\overline{\mathcal{K}}$ are nef classes. Then we have
\be\label{AF kahler class}(\alpha\cdot c\cdot c_3\cdots c_n)^2\geq
(\alpha^2\cdot c_3\cdots c_n)
(c^2\cdot c_3\cdots c_n),\ee
with equality if and only if $\alpha$ and $c$ are proportional: $\alpha=\lambda c$ with some $\lambda\in\mathbb{R}$.

By continuity, we have
\be\label{AF nef class}(\alpha\cdot\gamma\cdot\gamma_3\cdots\gamma_n)^2\geq
(\alpha^2\cdot\gamma_3\cdots\gamma_n)(\gamma^2\cdot\gamma_3
\cdots\gamma_n).\ee
\end{proposition}

\begin{remark}
\begin{enumerate}
\item
When the author wrote the article \cite{Li2}, he didn't notice the related results in \cite{De} and \cite{DP}. However, we shall see in establishing our first main result, Theorem \ref{AF complex mixed discriminant}, that the characterization of the equality case of (\ref{AF kahler class}) will play a substantial role.

\item
Similar to the proof of (\ref{AF mixed volume for m}), repeated use of (\ref{AF kahler class}) and the characterization of its equality case yield, for any $2\leq m\leq n$,
\be\label{AF kahler class for m}(c_1\cdots c_m\cdot c_{m+1}\cdots
c_n)^m\geq\prod_{i=1}^m(c_i^m\cdot c_{m+1}\cdots c_n),~\forall~ c_1,\ldots,c_n\in\mathcal{K},\ee
with equality if and only if the K\"{a}hler classes $c_1,c_2,\ldots,c_m$ are all proportional. Consequently by continuity the inequality (\ref{AF kahler class for m}) remains true when these classes $c_i$ are nef. These results were stated and some of their proofs were outlined in \cite[Remark 5.3]{De}. The algebraic setting of these results was established in \cite[\S 1.6]{La} by Lazarsfeld. We shall illustrate in Section \ref{last section} a detailed and clean proof of (\ref{AF kahler class for m}) assuming the validity of (\ref{AF kahler class}).

\item
With the above notation understood, the original Khovanskii-Teissier inequalities indeed read
\be\label{KT inequality}(\gamma_1^m\cdot\gamma_2^{n-m})^2\geq
(\gamma_1^{m+1}\cdot\gamma_2^{n-m-1})
(\gamma_1^{m-1}\cdot\gamma_2^{n-m+1}),\forall~1\leq m\leq n-1,~ \forall~\gamma_1,\gamma_2\in\overline{\mathcal{K}},\ee
which follows directly from (\ref{AF nef class}).
\end{enumerate}
\end{remark}

\section{Main results}\label{main results}
In this section we shall state our main results, Theorems \ref{AF complex mixed discriminant}, \ref{determinantal} and \ref{equality case1} as well as some of their consequences.

Let $A$ be a complex Hermitian $n\times n$ matrix, i.e., $A=\overline{A^t}$, where $``t"$ denotes the transpose of a matrix. $A$ is called \emph{positive definite} if $$(z^1,\ldots,z^n)\cdot A\cdot(\overline{z^1},\ldots,\overline{z^n})^t\geq 0,\qquad (z^1,\ldots,z^n)\in\mathbb{C}^n,$$
and with equality if and only if $z^1=\cdots=z^n=0$. Positive semi-definiteness can be similarly defined. In most related articles on mixed discriminants only real symmetric matrices are treated and only a few of them investigated complex Hermitian matrices, e.g., \cite{Ba} and \cite{Gu} and the related references therein. Gurvits proved in \cite[Theorem 5.2]{Gu} a generalized AF type inequality for positive semi-definite complex Hermitian matrices: by taking $\alpha=(1,\ldots,1)$, $\alpha_1=(2,0,1,\ldots,1)$, $\alpha_2=(0,2,1,\ldots,1)$, and $\gamma_1=\gamma_2=\frac{1}{2}$, equation (21) reduces to
the following expected inequality (\ref{AF complex mixed discriminant inequality}) with an extra factor $(\frac{n^n}{n!})^2$. Now our first main result is a geometrical proof of the following complex version of the AF type inequality for mixed discriminants, which should be known to some experts on inequalities of matrices but never appeared in existing literature.

\begin{theorem}[Complex version of AF type inequality for mixed discriminants]\label{AF complex mixed discriminant}
Let $A$, $B$, $A_3,\ldots,A_n$ be complex Hermitian $n\times n$ matrices, where $A, A_3,\ldots,A_n$ are positive definite and $B$ is arbitrary. Then
\be\label{AF complex mixed discriminant inequality}
D(A,B,A_3,\ldots,A_n)^2\geq D(A,A,A_3,\ldots,A_n)D(B,B,A_3,\ldots,A_n),\ee
with equality if and only if $A$ and $B$ are proportional: $B=\lambda A$ for some real number $\lambda$. By continuity, \rm{(}\ref{AF complex mixed discriminant inequality}\rm{)} remains true if these $A,A_3,\ldots,A_n$ are only positive semi-definite.
\end{theorem}

\begin{remark}
As we have mentioned in the Introduction, the validity of this result in the complex version should be known to some experts (at least to Alexandrov himself). But it seems that it is in our current article the first time to be given an explicit statement and proof. Further, our proof is not combinatorial but purely algebro-geometric.
\end{remark}

Similar to (\ref{AF mixed volume for m}) and Corollary \ref{AF mixed volume for BM}, we also have for the mixed discriminants the following two direct consequences/improvements.
\begin{corollary}\label{AF complex mixed for m}
Let $A_1,\ldots,A_n$ be $n\times n$ complex Hermitian positive definite matrices. Then for each $2\leq m\leq n$, we have
\be\label{AF complex mixed formula for m}
D(A_1,\ldots,A_m,A_{m+1},\ldots,A_n)^m\geq\prod_{i=1}^mD(A_i[m],A_{m+1},\ldots,A_n),\ee
where the equality holds if and only if the matrices $A_1,\ldots,A_m$ are all proportional.
Moreover, by continuity, the inequality (\ref{AF complex mixed formula for m}) remains true if these $A_i$ are only assumed to be positive semi-definite.
\end{corollary}

\begin{corollary}\label{AF complex mixed discriminant for BM}
Assume that $A_0,A_1,A_{m+1},\ldots,A_n$ ($2\leq m\leq n$) are positive semi-definite complex Hermitian $n\times n$ matrices. Define
\be\label{g(lambda)}g(\lambda):=D\Big(\big((1-\lambda)A_0+\lambda A_1\big)[m],A_{m+1},\ldots,A_n\Big)^{\frac{1}{m}},~0\leq\lambda\leq 1.\ee
Then the function $g(\lambda)$ is concave on $[0,1]:$
\be\label{Brunn-Minkowski for mixed discriminant}\begin{split}&D\Big(\big((1-\lambda)A_0+\lambda A_1\big)[m],A_{m+1},\ldots,A_n\Big)^{\frac{1}{m}}\\ \geq&(1-\lambda)D(A_{0}[m],A_{m+1},\ldots,A_n)^{\frac{1}{m}} +\lambda D(A_{1}[m],A_{m+1},\ldots,A_n)^{\frac{1}{m}},\qquad 0\leq\lambda\leq 1.\end{split}\ee

Furthermore, if the equality case in (\ref{Brunn-Minkowski for mixed discriminant}) holds for \emph{some} $\lambda\in(0,1)$, it must hold for \emph{all} $\lambda\in[0,1]$: $$g(\lambda)\equiv(1-\lambda)g(0)+\lambda g(1),\qquad\forall~\lambda\in[0,1].$$
\end{corollary}

We now state our second main result, which is inspired by a beautiful observation due to Shephard \cite{Sh}. Given a pair of positive integers $(p,n)$, we may apply the AF inequality (\ref{AF mixed volume}) in all possible ways to yield a set of quadratic inequalities satisfied by arbitrary $p$ convex bodies in $\mathbb{R}^n$. Shephard considered in \cite{Sh} if this set of inequalities is a full set for given $(p,n)$ and obtained several interesting results. Along this line, he proved a positive semi-definiteness of a matrix involved in the mixed volumes (\cite[(2.5)]{Sh}) and thus generalized the AF inequality (\ref{AF mixed volume}) to a determinantal case (\cite[(2.4)]{Sh}). What we notice is that Shephard's proof is only a \emph{formal} application of the inequality (\ref{AF mixed volume}) and can be extended to a general case, which is exactly our second main result:
\begin{theorem}\label{determinantal}
Let $C$ be a closed cone in a real vector space and $f: C\times C\rightarrow\mathbb{R}$ be a function which is symmetric: $f(u,v)=f(v,u)$, and bilinear in the following sense:
\be\label{determinallinear}f(\lambda^1u_1+\lambda^2u_2,\mu^1v_1+\mu^2v_2)
=\sum_{i,j=1}^2\lambda^i\mu^jf(u_i,v_j),\qquad\forall~\lambda^i,~\mu^i\geq0,~\forall~ u_i,~v_i\in C,\ee
and satisfies the AF type relation:
 \be\label{determinantalAFtype}\big[f(u,v)\big]^2\geq f(u,u)\cdot f(v,v),\qquad\forall~u,~v\in C.\ee
Arbitrarily choose $r+1$ elements $u_0,u_1,\ldots,u_r\in C$ $(r\geq 1)$ and denote $d_{ij}:=f(u_i,u_j)$ $(0\leq i,j\leq r).$ Then
the $r\times r$ matrix $$\big(d_{0i}d_{0j}-d_{00}d_{ij}\big)_{i,j=1}^r$$
is positive semi-definite. Furthermore,
\be\label{determinatalvalue}
0\leq\det\Big(\big(d_{0i}d_{0j}-d_{00}d_{ij}\big)_{i,j=1}^r\Big)=
(-1)^r\cdot d^{r-1}_{00}\cdot\det\Big(\big(d_{ij}\big)_{i,j=0}^r\Big).\ee
\end{theorem}

The following direct consequences of Theorem \ref{determinantal}, whose proof will be included for the reader's convenience in Section \ref{Proof of determinantal}, are generalizations of the AF type inequalities in Proposition \ref{AF kahler/nef} and Theorem \ref{AF complex mixed discriminant} to determinantal cases.

\begin{corollary}\label{coro1}
For arbitrarily choose positive integer $r$ and $n+r-1$ positive semi-definite complex Hermitian $n\times n$ matrices: $A_0,A_1,\ldots,A_r,B_3,\ldots,B_n$, denote by $$d_{ij}=D(A_i,A_j,B_3,\ldots,B_n),\qquad 0\leq i,j\leq r.$$ Then we have
\begin{eqnarray}\label{determinantal general mixed}
\left\{ \begin{array}{ll}
\det\Big(\big(d_{0i}d_{0j}-d_{00}d_{ij}\big)_{i,j=1}^r\Big)\geq0,\\
~\\
(-1)^r\det\Big(\big(d_{ij}\big)_{i,j=0}^r\Big)\geq0.\\
\end{array} \right.
\end{eqnarray}
\end{corollary}

\begin{corollary}\label{coro2}
Assume that $M$ is a compact K\"{a}hler manifold of complex dimension $n$. For arbitrarily choose positive integer $r$ and $n+r-1$ nef classes $\beta_0,\beta_1,\ldots,\beta_r,\gamma_3,\ldots,\gamma_n$ in $H^{1,1}(M;\mathbb{R})$, denote by \be\label{dij}d_{ij}=\beta_i\cdot\beta_j\cdot\gamma_3\cdots\gamma_n,\qquad (0\leq i,j\leq r).\ee
 Then we have
\begin{eqnarray}\label{determinantal general nef}
\left\{ \begin{array}{ll}
\det\Big(\big(d_{0i}d_{0j}-d_{00}d_{ij}\big)_{i,j=1}^r\Big)\geq0,\\
~\\
(-1)^r\det\Big(\big(d_{ij}\big)_{i,j=0}^r\Big)\geq0.\\
\end{array} \right.
\end{eqnarray}
\end{corollary}

\begin{remark}
Inequalities (\ref{determinantal general mixed}) and (\ref{determinantal general nef}) reduce to the original AF type inequalities when taking $r=1$. The $r=2$ case of the inequality (\ref{determinantal general nef}) shall play a key role in the proof of our central result in this article, Theorem \ref{equality case1}.
\end{remark}


We now move to \emph{our central result} in this article. Recall that the characterization of the equality cases of the inequalities (\ref{AF kahler class}) and (\ref{AF kahler class for m}) are clear when the classes involved are K\"{a}hler: the equality cases hold if and only if the K\"{a}hler classes are proportional. This is \emph{not} the case when the classes involved in the inequalities (\ref{AF nef class}) and (\ref{AF kahler class for m}) are only assumed to be nef and it turns out to be a quite difficult problem. This characterization problem of the equality case was first proposed and studied by Teissier (\cite[p. 139]{Te3}) for his inequalities (\ref{KT inequality}), where only two nef classes are involved. Inspired by this, a more general characterization problem has been posed by the author in \cite[Question 3.3]{Li2}. When these two classes are nef and big (the notion of bigness will be reviewed in Section \ref{Proof of equality}), Teissier's problem was solved by Boucksom, Favre and Jonsson for projective manifolds (\cite[Theorem D]{BFJ}), as an application of their differentiability theorem in \cite{BFJ}, and then solved by Fu and Xiao for general compact K\"{a}hler manifolds (\cite{FX2}): just as expected, for two nef and big classes $\gamma_1$ and $\gamma_2$ \emph{all} the equalities in (\ref{KT inequality}) hold if and only if $\gamma_1$ and $\gamma_2$ are proportional. Note that the equality case stated in \cite[Theorem D]{BFJ} is, in our notation, $$(\gamma_1^{n-1}\gamma_2)^n=(\gamma_2^n)\cdot(\gamma_1^n)^{n-1},$$ which is equivalent to the validity of \emph{all} the equality cases in (\ref{KT inequality}) (cf. for instance, the equivalent statements of (1) and (3) in (\cite[Theorem
2.1]{FX2} and its simple proof in \cite[p. 3]{FX2}). Recently in \cite[Theorem 2.1]{LX1} Lehmann and Xiao, based on the ideas in \cite{FX1} and \cite{FX2}, solved the characterization problem for the equality case $m=n$ in (\ref{AF kahler class for m}), still assuming the classes involved are nef and big.

With these backgrounds in mind, now comes \emph{our central result} in this article, which completely settle the equality characterization problem in inequalities (\ref{AF kahler class}) and (\ref{AF kahler class for m}) for nef and big classes.
\begin{theorem}\label{equality case1}
Let $\gamma_1,\gamma_2,\ldots,\gamma_n$ be $n$ nef and big real $(1,1)$ cohomology classes on a compact K\"{a}hler manifold $M$ of complex dimension $n$. Then the following two assertions are equivalent:
\begin{enumerate}
\item
\be\label{equality 2}(\gamma_1\cdot\gamma_2\cdot\gamma_3\cdots\gamma_n)^2=
(\gamma_1^2\cdot\gamma_3\cdots\gamma_n)\cdot(\gamma_2^2\cdot\gamma_3\cdots\gamma_n);\ee

\item
the two $(n-1,n-1)$ real cohomology classes $\gamma_1\wedge\gamma_3\wedge\cdots\wedge\gamma_n$ and $\gamma_2\wedge\gamma_3\wedge\cdots\wedge\gamma_n$ are proportional in $H^{n-1,n-1}(M;\mathbb{R})$.
\end{enumerate}
\end{theorem}

Repeated use of Theorem \ref{equality case1} leads to the following improvement.
\begin{theorem}\label{equality case2}
Let $\gamma_1,\gamma_2,\ldots,\gamma_n$ be $n$ nef and big real $(1,1)$ cohomology classes on a compact K\"{a}hler manifold $M$ of complex dimension $n$ and $2\leq m\leq n$. Then the following two assertions are equivalent:
\begin{enumerate}
\item
\be\label{equality m}
(\gamma_1\cdots\gamma_m\cdot\gamma_{m+1}\cdots\gamma_n)^m=\prod_{i=1}^m
(\gamma_i^m\cdot\gamma_{m+1}\cdots\gamma_n);\ee

\item
the $(n-1,n-1)$ real cohomology classes $$\gamma_{i_1}\wedge\gamma_{i_2}\wedge\cdots\wedge
\gamma_{i_{m-1}}\wedge\gamma_{m+1}\wedge\cdots\wedge\gamma_n\qquad(1\leq i_1,\ldots,i_{m-1}\leq m)$$
 are all proportional in $H^{n-1,n-1}(M;\mathbb{R})$.
\end{enumerate}
\end{theorem}

Theorem \ref{equality case2}, together with a beautiful injectivity result due to Fu and Xiao in \cite{FX1}, yields the above-mentioned related results in \cite{BFJ}, \cite{FX2} and \cite{LX1}.

\begin{corollary}[Lehmann-Xiao, Fu-Xiao]\label{equality case3}
Let $\gamma_1,\gamma_2,\ldots,\gamma_n$ be $n$ nef and big real $(1,1)$ cohomology classes on a compact K\"{a}hler manifold $M$ of complex dimension $n$. Then the equality
\be\label{equality n}
(\gamma_1\cdot\gamma_2\cdots\gamma_n)^n=\prod_{i=1}^n
(\gamma_i^n)
\ee
holds if and only if these classes $\gamma_1,\ldots,\gamma_n$ are all proportional in $H^{1,1}(M;\mathbb{R})$. In particular, setting $\gamma_1=\gamma_2=\cdots=\gamma_{n-1}$ leads to the related results in \cite[Theorem 2.1]{FX2} and \cite[Theorem D]{BFJ}.
\end{corollary}

\begin{proof}
According to Theorem \ref{equality case2}, the equality (\ref{equality n}) holds if and only if the $(n-1,n-1)$ cohomology classes
$$\gamma_{i_1}\wedge\gamma_{i_2}\wedge\cdots\wedge\gamma_{i_{n-1}},\qquad 1\leq i_1,\ldots,i_{n-1}\leq n,$$
are all proportional. In particular, the cohomology classes $\gamma_1^{n-1},\gamma^{n-1}_2\ldots,\gamma_n^{n-1}$ are all proportional. Then an injectivity result for nef and big classes due to Fu and Xiao (\cite[Theorems 1.2, 2.4]{FX1}) tells us that these $\gamma_1,\gamma_2,\ldots,\gamma_n$ themselves must be proportional.
\end{proof}

\begin{remark}
\begin{enumerate}
\item
In contrast to the special cases in Corollary \ref{equality case3} obtained by Lehmann, Fu, Xiao etc., in the general case in Theorem \ref{equality case1} the proportionality of $\gamma_1\wedge\gamma_3\wedge\cdots\wedge\gamma_n$ and $\gamma_2\wedge\gamma_3\wedge\cdots\wedge\gamma_n$ is \emph{not} enough to derive that of $\gamma_1$ and $\gamma_2$, even if the class $\gamma_3\wedge\cdots\wedge\gamma_n$ contains a strictly positive form in the strong sense (cf. \cite[Remark 2.9]{DN2}).

\item
The special case of the injectivity result of Fu and Xiao for K\"{a}hler classes (\cite[Prop. 1.1]{FX1}) can also be obtained by a result of Hard Lefschetz type property due to Ross and Toma (\cite[Coro. 8.6]{RT}).
\end{enumerate}
\end{remark}

\section{Proof of Theorem \ref{AF complex mixed discriminant}}\label{Proof of AF complex mixed discriminant}
In this section we shall
apply Proposition \ref{AF kahler/nef} to prove Theorem \ref{AF complex mixed discriminant}.

\begin{proof}
First note that each complex Hermitian $n\times n$ matrix
$A$ corresponds to a constant real $(1,1)$-form on $\mathbb{C}^n$ under the global coordinate system $(z^1,\ldots,z^n)$ as follows:
$$A=\big(a_{ij}\big)_{i,j=1}^n\longleftrightarrow\alpha(A):=
\sqrt{-1}\sum_{i,j=1}^na_{ij}\text{d}z^i\wedge\text{d}\overline{z^j}.$$

By definition $\alpha(A)$ is a K\"{a}hler form if and only if $A$ is positive definite. Moreover, their wedges satisfy
\be\label{wedge}
\begin{split}
\alpha(A_1)\wedge\alpha(A_2)\wedge\cdots\wedge\alpha(A_n)&
=\bigwedge_{k=1}^n\big(\sqrt{-1}\sum_{i,j=1}^na_{ij}^{(k)}
\text{d}z^i\wedge\text{d}\overline{z^j}\big)
\qquad\Big(A_k:=\big(a_{ij}^{(k)}\big)_{i,j=1}^n\Big)\\
&=\Big(\sum_{\sigma\in S_n}\begin{vmatrix}
a^{(\sigma(1))}_{11}&\cdots&a^{(\sigma(n))}_{1n}\\
\vdots&\ddots&\vdots\\
a^{(\sigma(1))}_{n1}&\cdots&a^{(\sigma(n))}_{nn}
\end{vmatrix}\Big)\cdot\bigwedge_{i=1}^n\big(\sqrt{-1}\text{d}z^i
\wedge\text{d}\overline{z^i}\big)\\
&=n!\cdot D(A_1,\ldots,A_n)\cdot2^n\cdot\text{d}V(\mathbb{C}^n),
\end{split}
\ee
where $$\text{d}V(\mathbb{C}^n)=\text{d}x^1\wedge\text{d}y^1\wedge
\cdots\wedge\text{d}x^n\wedge\text{d}y^n\qquad (z^i:=x^i+\sqrt{-1}y^i)$$ is the volume form of $\mathbb{C}^n$ with respect to the standard Euclidean metric.

With this simple but crucial observation in mind, we can now proceed to prove Theorem \ref{AF complex mixed discriminant}. Let $B,A,A_3,\ldots,A_n$ be the complex Hermitian matrices as in Theorem \ref{AF complex mixed discriminant}. Then in our notation $\alpha(B),\alpha(A),\alpha(A_3),\ldots,\alpha(A_n)$ are constant real $(1,1)$-forms on $\mathbb{C}^n$ and among them the latter $n-1$ forms are K\"{a}hler. They are constant forms and thus can descend to the complex torus $$\mathbb{C}^n\big/(\mathbb{Z}+\sqrt{-1}\mathbb{Z})^n=:T,$$ which, for simplicity, still denote on $T$ by $\alpha(B),\alpha(A),\alpha(A_3),\ldots,\alpha(A_n)$. This means that, among the induced real $(1,1)$ cohomology classes
$$\big[\alpha(B)\big],~\big[\alpha(A)\big],~\big[\alpha(A_3)\big],~\ldots,~[\alpha(A_n)]$$ on $T$, the latter $n-1$ classes are K\"{a}hler. Thus applying the inequality (\ref{AF kahler class}) to them yields
\be\label{1}
\begin{split}
&\Big(\big[\alpha(B)\big]\cdot\big[\alpha(A)\big]\cdot\big[
\alpha(A_3)\big]\cdots\big[\alpha(A_n)\big]\Big)^2\\
\geq&\Big(\big[\alpha(B)\big]\cdot\big[\alpha(B)\big]\cdot\big[
\alpha(A_3)\big]\cdots\big[\alpha(A_n)\big]\Big)\cdot
\Big(\big[\alpha(A)\big]\cdot\big[\alpha(A)\big]\cdot\big[
\alpha(A_3)\big]\cdots\big[\alpha(A_n)\big]\Big),
\end{split}\ee
with equality if and only if the two cohomology classes $\big[\alpha(B)\big]$ and $\big[\alpha(A)\big]$ are proportional: $$\big[\alpha(B)\big]=\lambda\big[\alpha(A)\big]\in H^{1,1}(T;\mathbb{R})$$
 with some real number $\lambda$. Note that
\be\label{2}\begin{split}
\big[\alpha(B)\big]\cdot\big[\alpha(A)\big]\cdot\big[\alpha(A_3)\big]
\cdots\big[\alpha(A_n)\big]&=
\int_T\alpha(B)\wedge\alpha(A)\wedge\alpha(A_3)\wedge\cdots\wedge\alpha(A_n)\\
&=\int_T\big(n!\cdot D(B,A,A_3,\ldots,A_n)\cdot2^n\cdot\text{d}V(T)
\qquad\big(\text{by (\ref{wedge})}\big)\\
&=n!\cdot2^n\cdot D(B,A,A_3,\ldots,A_n)\cdot\text{Vol}(T),\end{split}\ee
where $\text{Vol}(T)$ is the volume of $T$ with the standard metric induced from $\mathbb{C}^n$. Substituting (\ref{2}) into (\ref{1}) yields exactly the desired inequality (\ref{AF complex mixed discriminant inequality}) and it now suffices to characterize its equality case.

By the $\partial\bar{\partial}$-lemma on compact K\"{a}hler manifolds, $\big[\alpha(B)\big]=\lambda\big[\alpha(A)\big]$ is equivalent to the existence of a smooth real function $f$ on $T$, unique up to an additive constant, such that
\be\label{3}
\sqrt{-1}\sum_{i,j=1}^n(b_{ij}-\lambda a_{ij})
\text{d}z^i\wedge\text{d}\overline{z^j}=
\sqrt{-1}\partial\bar{\partial}f,\qquad \Big(A:=\big(a_{ij}\big)_{i,j=1}^n,~B:=\big(b_{ij}\big)_{i,j=1}^n\Big).
\ee

Taking trace with respect to the standard metric on $T$ on both sides of (\ref{3}) yields $$\Delta f=\text{constant},\qquad\text{$\Delta$: the Laplacian operator}.$$
This, together with the facts of compactness and connectedness of $T$, leads to the fact that $f$ itself be a constant and thus $B=\lambda A$, which completes the proof of Theorem \ref{AF complex mixed discriminant}.
\end{proof}

\section{Proofs of Theorem \ref{determinantal} and its corollaries}\label{Proof of determinantal}
We shall prove in this section Theorem \ref{determinantal} as well as its two corollaries: Corollaries \ref{coro1} and \ref{coro2}.

\subsection{Proof of Theorem \ref{determinantal}}
The positive semi-definiteness of the matrix $$\big(d_{0i}d_{0j}-d_{00}d_{ij}\big)_{i,j=1}^r$$ is equivalent to
\be\label{positive semi}\sum_{i,j=1}^r\lambda^i\lambda^j
(d_{0i}d_{0j}-d_{00}d_{ij})\geq0,\qquad\forall~(\lambda^1,
\ldots,\lambda^r)\in\mathbb{R}^r.\ee

The idea of the following proof of (\ref{positive semi}) essentially is due to \cite[p. 133]{Sh}, but with some simplifications and filling in some necessary details.

We arbitrarily fix  $(\lambda^1,\ldots,\lambda^r)\in\mathbb{R}^r$ and let $t>0$ be an indeterminate. Without loss of generality, we may assume that
$$\lambda^i\geq 0,~1\leq i\leq s;\qquad\lambda^j\leq 0,~s+1\leq j\leq r.$$
For simplicity in the sequel we shall assume that
$$1\leq i,i_1,i_2\leq s,\qquad s+1\leq j,j_1,j_2\leq r$$
and apply the Einstein summand convention for these indices. Under the assumption of (\ref{determinantalAFtype}) we have
\be\label{4}
\begin{split}&\big[f(u_0+t\lambda^iu_i,\frac{1}{t}u_0
-\lambda^ju_j)\big]^2\\
-&f(u_0+t\lambda^{i_1}u_{i_1},u_0+t\lambda^{i_2}u_{i_2})\cdot
f(\frac{1}{t}u_0-\lambda^{j_1}u_{j_1},\frac{1}{t}u_0-
\lambda^{j_2}u_{j_2})\geq0.\end{split}\ee
Clearly the reason for separating the positive and negative coefficients $\lambda^i$ and $\lambda^j$ is to make sure that the elements $u_0+t\lambda^iu_i$ and $\frac{1}{t}u_0
-\lambda^ju_j$ remain in the cone where the AF inequality (\ref{determinantalAFtype}) holds.

Applying bilinear property (\ref{determinallinear}) and the symmetry of $f(\cdot,\cdot)$ to (\ref{4}) yield
\be\label{5}\begin{split}
&(\frac{1}{t}d_{00}+\lambda^id_{0i}-\lambda^jd_{0j}-
t\lambda^i\lambda^jd_{ij})^2\\
-&
(d_{00}+2t\lambda^{i}d_{0i}
+t^2\lambda^{i_1}\lambda^{i_2}d_{i_1i_2})
\cdot
(\frac{1}{t^2}d_{00}-\frac{2}{t}\lambda^{j}d_{0j}
+\lambda^{j_1}\lambda^{j_2}d_{j_1j_2})\geq0.\end{split}\ee

The constant term on the LHS of (\ref{5}) is
\be\begin{split}
&\big[(\lambda^id_{0i}-\lambda^jd_{0j})^2
-2\lambda^i\lambda^jd_{00}d_{ij}\big]
-\big[\lambda^{j_1}\lambda^{j_2}d_{00}d_{j_1j_2}
-4\lambda^i\lambda^jd_{0i}d_{0j}+
\lambda^{i_1}\lambda^{i_2}d_{00}d_{i_1i_2}\big]\\
=&(\lambda^{i_1}\lambda^{i_2}d_{0i_1}d_{0i_2}+\lambda^{j_1}\lambda^{j_2}d_{0j_1}d_{0j_2}
-2\lambda^i\lambda^jd_{0i}d_{0j}-2\lambda^i\lambda^jd_{00}d_{ij})\\
-&(\lambda^{j_1}\lambda^{j_2}d_{00}d_{j_1j_2}
-4\lambda^i\lambda^jd_{0i}d_{0j}+\lambda^{i_1}\lambda^{i_2}d_{00}d_{i_1i_2})\\
=&\lambda^{i_1}\lambda^{i_2}(d_{0i_1}d_{0i_2}-d_{00}d_{i_1i_2})+
\lambda^{j_1}\lambda^{j_2}(d_{0j_1}d_{0j_2}-d_{00}d_{j_1j_2})
+2\lambda^i\lambda^j(d_{0i}d_{0j}-d_{00}d_{ij})\\
=&\sum_{p,q=1}^r\lambda^p\lambda^q(d_{0p}d_{0q}-d_{00}d_{pq}).
\end{split}\nonumber\ee

It is easy to check that the coefficients in front of $t^{-1}$ and $t^{-2}$ on the LHS of (\ref{5}) vanish. Therefore (\ref{5}) becomes
\be\label{6}(\cdots)t+(\cdots)t^2+\sum_{p,q=1}^r\lambda^p\lambda^q(d_{0p}d_{0q}-d_{00}d_{pq})\geq0.\ee
Here the two $(\cdots)$ denote respectively the coefficients in front of $t$ and $t^2$, with whose concrete values we are not concerned. Letting $t$ tend to $0$ in (\ref{6}) yields the desired inequality (\ref{positive semi}).

Next we prove (\ref{determinatalvalue}). Let
\begin{eqnarray}
\left\{ \begin{array}{ll}
\vec{\theta}:=(d_{01},\ldots,d_{0r})^t\\
~\\
\big(d_{ij}\big)_{i,j=1}^r:=(\vec{\theta}_1,\ldots,\vec{\theta}_r).\\
\end{array} \right.\nonumber
\end{eqnarray}

Then
\be\label{7}\begin{split}
&\det\Big(\big(d_{0i}d_{0j}-d_{00}d_{ij}\big)_{i,j=1}^r\Big)\\
=&\det\big(d_{01}\vec{\theta}-d_{00}\vec{\theta}_1,
d_{02}\vec{\theta}-d_{00}\vec{\theta}_2,\ldots,
d_{0r}\vec{\theta}-d_{00}\vec{\theta}_r\big)\\
=&(-1)^r\cdot\det\big(d_{00}\vec{\theta}_1-d_{01}\vec{\theta},
d_{00}\vec{\theta}_2-d_{02}\vec{\theta},\ldots,
d_{00}\vec{\theta}_r-d_{0r}\vec{\theta}\big)\\
=&(-1)^r\cdot d_{00}^{r-1}\cdot\Big[d_{00}\cdot\det(\vec{\theta}_1,\ldots,\vec{\theta}_r)
-\sum_{i=1}^r\big[d_{0i}\cdot\det(\vec{\theta}_1,\ldots,
\vec{\theta}_{i-1},\vec{\theta},\vec{\theta}_{i+1},\ldots,\vec{\theta}_r)\big]\Big]\\
=&(-1)^r\cdot d_{00}^{r-1}\cdot\Big[d_{00}\cdot\det(\vec{\theta}_1,\ldots,\vec{\theta}_r)
+\sum_{i=1}^r\big[d_{0i}\cdot(-1)^i\cdot\det(\vec{\theta},\vec{\theta}_1,\ldots,
\vec{\theta}_{i-1},\vec{\theta}_{i+1},\ldots,\vec{\theta}_r)\big]\Big]\\
=&(-1)^r\cdot d_{00}^{r-1}\cdot\det\Big(\big(d_{ij}\big)_{i,j=0}^r\Big).
\end{split}\ee

The reason for the last equality in (\ref{7}) is due to the fact that the expression inside $[\cdots]$ in the last but one line in (\ref{7}),
under the assumptions $d_{ij}=d_{ji}$, is nothing but the expansion of $\det\Big(\big(d_{ij}\big)_{i,j=0}^r\Big)$ along its first line $(d_{00},d_{01},\cdots,d_{0r})$. This completes the proof of Theorem \ref{determinantal}.

\subsection{Proofs of Corollaries \ref{coro1} and \ref{coro2}}
Corollaries \ref{coro1} and \ref{coro2} are essentially direct consequences of Theorem \ref{determinantal}. For the reader's convenience we still indicate that how they can be derived from Theorem \ref{determinantal}.

Note that the set of positive semi-definite complex Hermitian $n\times n$ matrices can be viewed as a closed cone, denoted it by $\mathcal{M}$, in $\mathbb{R}^{n^2}$. Then, under the assumptions of Corollary \ref{coro1}, the function $f$ can be defined by $$f(A,B):=D(A,B,B_3,\ldots,B_n):~
\mathcal{M}\times\mathcal{M}\longrightarrow\mathbb{R}_{\geq0}.$$
The reason that this $f$ is indeed nonnegative is well-known (\cite[Lemma 2]{Ba}) and moreover $f$ is strictly positive if these $A,B,B_3,\ldots,B_n$ are all positive definite (\cite[Theorem 9]{Ba}).

Now applying Theorem \ref{determinantal} to this situation yields the first inequality in (\ref{determinantal general mixed}): $$\det\Big(\big(d_{0i}d_{0j}-d_{00}d_{ij}\big)_{i,j=1}^r\Big)\geq 0.$$

If those $A_0,A_1,\ldots,A_r,B_3,\ldots,B_3$ in Corollary \ref{coro1} are positive definite, then $d_{00}>0$ (\cite[Theorem 9]{Ba}). This, together with the first inequality in (\ref{determinantal general mixed}), tells us that the second one in (\ref{determinantal general mixed}) is true if the matrices involved are positive definite. However, this is enough to derive the desired result as positive semi-definite matrices can be approximated by positive definite ones.

The proof of Corollary \ref{coro2} is identically the same as that of Corollary \ref{coro1}: apply Theorem \ref{determinantal} to the closed nef cone $\overline{\mathcal{K}}\subset H^{1,1}(M;\mathbb{R})$ to yield the first one in (\ref{determinantal general nef}). Note that $c_1\cdot c_2\cdot c_3\cdots c_n>0$ if these $c_i$ are all K\"{a}hler classes,and then the second one in (\ref{determinantal general nef}) holds for K\"{a}hler classes. Then the general case is also obtained by approximation.

\section{Proof of Theorem \ref{equality case1}}\label{Proof of equality}
In this section we shall prove Theorems \ref{equality case1}. The main ingredients in it are the case of $r=2$ in Corollary \ref{coro2}, and a result due to Dinh-Nguy\^{e}n in \cite{DN}, which let us first recall in what follows.

The main contributions in \cite{DN} are to extend the usual Hodge-Riemann bilinear theorem, the Hard Lefschetz theorem and the Lefschetz decomposition theorem to the mixed version by replacing a single K\"{a}hler class with possibly distinct K\"{a}hler classes (\cite[Theorems A-C]{DN}, \cite{DN2}). In addition to these, they also give some information on the cone of smooth strictly positive classes in $H^{n-2,n-2}(M,\mathbb{R})$ satisfying the Hodge-Riemann bilinear theorem (\cite[Prop. 4.1]{DN}). It is this result that play a key role in our proof of Theorem \ref{equality case1}.

Let $\omega$ be a K\"{a}hler class. For each $\Omega\in H^{n-2,n-2}(M;\mathbb{R})$, define
\be\label{primitivedef}P^{1,1}(M;\Omega,\omega):=\big\{\alpha\in
H^{1,1}(M,\mathbb{C})~\big|~ \alpha\wedge\Omega\wedge\omega=0\big\}\ee
and a bilinear form
$Q_{\Omega}(\cdot,\cdot)$ with respect to $\Omega$ on
$H^{1,1}(M,\mathbb{C})$ by
\be\label{quadric}Q_{\Omega}(\alpha,\beta)
:=-\int_M\alpha\wedge\bar{\beta}\wedge\Omega,
\qquad \alpha,~\beta\in H^{1,1}(M,\mathbb{C}).\ee

The usual Hodge-Riemann bilinear theorem tells us that  $Q_{\omega^{n-2}}(\cdot,\cdot)$ is positive-definite on $P^{1,1}(M;\omega^{n-2},\omega)$ and the mixed version due to Dinh-Nguy\^{e}n says that it remains true if $\Omega$ is the product of arbitrary $n-2$ K\"{a}hler classes $\omega_1\wedge\cdots\wedge\omega_{n-2}$: $Q_{\omega_1\wedge\cdots\wedge\omega_{n-2}}(\cdot,\cdot)$ is positive-definite on $P^{1,1}(M;\omega_1\wedge\cdots\wedge\omega_{n-2},\omega)$ (\cite[Theorem A]{DN}).

Following \cite{DN}, we define the cone
\be
\begin{split}
\mathcal{K}_{n-2}^{\text{HR}}(\omega):=\big\{&\Omega\in H^{n-2,n-2}(M,\mathbb{R}):~\text{classes of smooth strictly positive}\\
 &\text{$(n-2,n-2)$-forms such that $Q_{\Omega}(\cdot,\cdot)$ are positive-definite on $P^{1,1}(M;\Omega,\omega)$}\big\}
\end{split}\nonumber\ee
and $\overline{\mathcal{K}}_{n-2}^{\text{HR}}(\omega)$ its closure.
Here positivity of forms of higher bidegrees can be understood in the weak or strong sense, as stated in \cite[p. 847]{DN}. In any case by the mixed version of the Hodge-Riemann bilinear theorem in \cite{DN} we have
$$\big\{\omega_1\wedge\cdots\wedge\omega_{n-2}~\big|~\text{
$\omega_i$ are K\"{a}hler classes}\big\}\subset \mathcal{K}_{n-2}^{\text{HR}}(\omega)$$
and thus
\be\label{nef HR}
\big\{\gamma_1\wedge\cdots\wedge\gamma_{n-2}~\big|~\text{
$\gamma_i$ are nef classes}\big\}\subset \overline{\mathcal{K}}_{n-2}^{\text{HR}}(\omega)
.\ee

With the above notation understood, we have the following result (\cite[Prop. 4.1]{DN}), only whose second part shall be used in our proof.
\begin{proposition}[Dinh-Nguy\^{e}n]\label{DN}
\begin{enumerate}
\item
$\mathcal{K}_{n-2}^{\text{HR}}(\omega)$ \big(and hence $\overline{\mathcal{K}}_{n-2}^{\text{HR}}(\omega)$\big) does not depend on the K\"{a}hler class $\omega$ and thus can be simply denoted by $\mathcal{K}_{n-2}^{\text{HR}}$ and $\overline{\mathcal{K}}_{n-2}^{\text{HR}}$.

\item
Let $\Omega\in\overline{\mathcal{K}}_{n-2}^{\text{HR}}$ and $\omega$ be any K\"{a}hler class. Then $Q_{\Omega}(\cdot,\cdot)$ is positive semi-definite on $P^{1,1}(M;\Omega,\omega)$ and, for $\alpha\in P^{1,1}(M;\Omega,\omega)$ we have $Q_{\Omega}(\alpha,\alpha)=0$ if and only if $\alpha\wedge\Omega=0$.
\end{enumerate}
\end{proposition}

Besides this proposition, the $r=2$ case of our Corollary \ref{coro2} is another main ingredient in the proof of Theorem \ref{equality case1}, and so we rephrase it in the following proposition for our later reference.

\begin{proposition}\label{coro3}
Assume that $M$ is a compact K\"{a}hler manifold of complex dimension $n$. Arbitrarily choose $n+1$ nef classes $\beta_0,\beta_1,\beta_2,\gamma_3,\ldots,\gamma_n$ in $H^{1,1}(M;\mathbb{R})$. Denote by \be\label{dij}d_{ij}=\beta_i\cdot\beta_j\cdot\gamma_3\cdots\gamma_n,\qquad (0\leq i,j\leq 2).\ee
 Then we have
\be\label{r=2}(d_{01}^2-d_{00}d_{11})(d_{02}^2-d_{00}d_{22})
\geq(d_{01}d_{02}-d_{00}d_{12})^2.\ee
\end{proposition}

Before proceeding to prove Theorem \ref{equality case1}, we still need the following lemma, which is a well-known fact.
\begin{lemma}\label{lemma}
Let $T\in H^{n-1,n-1}(M;\mathbb{R})$ be an arbitrary $(n-1,n-1)$ cohomology class and $\omega$ any K\"{a}hler class. If $\int_MT\wedge\omega=0$, then $T\wedge\omega=0$.
\end{lemma}
\begin{proof}
Applying the usual Lefschetz decomposition theorem to $T$ with respect to the K\"{a}hler class $\omega$ yields $$T=(\lambda\omega+\alpha_{\text{p}})\wedge\omega^{n-2},$$ where $\lambda\in\mathbb{R}$ and $\alpha_{\text{p}}\in H^{1,1}(M;\mathbb{R})$ is a \emph{primitive} element with respect to $\omega$, i.e., $\alpha_{\text{p}}\wedge\omega^{n-1}=0$. Then the condition $\int_MT\wedge\omega=0$ tells us that $\lambda=0$ and hence $$T\wedge\omega=\alpha_{\text{p}}\wedge\omega^{n-1}=0.$$
\end{proof}

Now we are ready to prove Theorem \ref{equality case1}.

\begin{proof}
It suffices to show the implication $(1)\Rightarrow(2)$ in Theorem \ref{equality case1} as the implication $(2)\Rightarrow(1)$ is obvious.

Recall that a class $\gamma\in H^{1,1}(M;\mathbb{R})$ is called \emph{big} if there exists a K\"{a}hler current in it. It turns out that a nef class $\gamma$ is big if and only if its self-intersection number $\gamma^n>0$ (\cite[Theorem 0.5]{DP1}). If $\gamma_1,\gamma_2,\ldots,\gamma_n$ are nef and big classes then the intersection number $\gamma_1\cdot\gamma_2\cdots\gamma_n>0$.

Assume now that the equality (\ref{equality 2}) holds. Choose a K\"{a}hler class $\omega$ and apply the inequality (\ref{r=2}) in Proposition \ref{coro3} by setting $\beta_0=\gamma_1,$ $\beta_1=\gamma_2,$ $\beta_2=\omega$ and the $d_{ij}$ ($0\leq i,j\leq 2$) as in (\ref{dij}). The equality (\ref{equality 2}) under this notation reads $d_{01}^2=d_{00}d_{11}$, which, together with the inequality (\ref{r=2}), implies that $d_{01}d_{02}=d_{00}d_{12}$. This means that
\be\label{8}(\gamma_1\cdot\gamma_2\cdot\gamma_3\cdots\gamma_n)
(\gamma_1\cdot\omega\cdot\gamma_3\cdots\gamma_n)
=(\gamma_1^2\cdot\gamma_3\cdots\gamma_n)(\gamma_2\cdot\omega\cdot\gamma_3\cdots\gamma_n).\ee

Note that the four intersection numbers in (\ref{8}) are all positive as the classes involved are either nef and big or K\"{a}hler. Set
\be\label{8.5}\lambda:=\frac{\gamma_1\cdot\gamma_2\cdot\gamma_3\cdots\gamma_n}
{\gamma_1^2\cdot\gamma_3\cdots\gamma_n}>0\ee
and (\ref{8}) then becomes
\be\label{9}(\gamma_2-\lambda\gamma_1)\cdot\gamma_3\cdots
\gamma_n\cdot\omega=0.\ee

Combining (\ref{9}) with Lemma \ref{lemma} leads to $$(\gamma_2-\lambda\gamma_1)\wedge\gamma_3\wedge\cdots
\wedge\gamma_n\wedge\omega=0,$$
which implies, under the notion (\ref{primitivedef}), that
\be\label{10}(\gamma_2-\lambda\gamma_1)\in P^{1,1}(M;\gamma_3\wedge\cdots
\wedge\gamma_n,\omega).\ee

Now
\be\label{12}\begin{split}
&Q_{\gamma_3\wedge\cdots
\wedge\gamma_n}(\gamma_2-\lambda\gamma_1,\gamma_2-\lambda\gamma_1)\\
=&-(\gamma_2-\lambda\gamma_1)^2\cdot\gamma_3\cdots
\gamma_n\\
=&-(\gamma_2^2\cdot\gamma_3\cdots
\gamma_n)+2\lambda(\gamma_1\cdot\gamma_2\cdot\gamma_3\cdots
\gamma_n)-\lambda^2(\gamma_1^2\cdot\gamma_3\cdots
\gamma_n)\\
=&\frac{(\gamma_1\cdot\gamma_2\cdot\gamma_3\cdots
\gamma_n)^2-(\gamma_1^2\cdot\gamma_3\cdots
\gamma_n)(\gamma_2^2\cdot\gamma_3\cdots
\gamma_n)}{\gamma_1^2\cdot\gamma_3\cdots
\gamma_n}\qquad\big(\text{by (\ref{8.5})}\big)\\
=&0.\qquad\big(\text{by the assumption condition (\ref{equality 2})}\big)
\end{split}\ee

Note that (\ref{nef HR}) tells us that \be\label{11}\gamma_3\wedge\cdots
\wedge\gamma_n\in\overline{\mathcal{K}}_{n-2}^{\text{HR}}.\ee

Applying Proposition \ref{DN} under the conditions (\ref{10}), (\ref{12}) and (\ref{11}) yields the desired conclusion:
$$(\gamma_2-\lambda\gamma_1)\wedge\gamma_3\wedge\cdots\wedge\gamma_n=0,$$
i.e., the two $(n-1,n-1)$ real cohomology classes $\gamma_1\wedge\gamma_3\wedge\cdots\wedge\gamma_n$ and $\gamma_2\wedge\gamma_3\wedge\cdots\wedge\gamma_n$ are proportional in $H^{n-1,n-1}(M;\mathbb{R})$.
\end{proof}

\section{Appendix}\label{last section}
The ideas of the proofs of (\ref{AF mixed volume for m}), (\ref{AF kahler class for m}), (\ref{AF complex mixed formula for m}) and Theorem \ref{equality case2} are all via induction arguments and it should be a classical method and well-known to related experts. For the reader's convenience we shall give in this Appendix a proof of (\ref{AF kahler class for m}) and Theorem \ref{equality case2} under the assumptions of Proposition \ref{AF kahler/nef} and Theorem \ref{equality case1}, from which we can also see how to derive (\ref{AF mixed volume for m}) and (\ref{AF complex mixed formula for m}) from (\ref{AF mixed volume}) and (\ref{AF complex mixed discriminant inequality}) respectively.

\begin{theorem}\label{appendix}
Let $\gamma_1,\gamma_2,\ldots,\gamma_n$ be $n$ nef classes on a compact K\"{a}hler manifold $M$ of complex dimension $n$ and $2\leq m\leq n$. Then we have
\begin{enumerate}
\item
\be\label{inequality m}
(\gamma_1\cdots\gamma_m\cdot\gamma_{m+1}\cdots\gamma_n)^m\geq\prod_{i=1}^m
(\gamma_i^m\cdot\gamma_{m+1}\cdots\gamma_n);\ee

\item
if these $\gamma_1,\ldots,\gamma_n$ are K\"{a}her classes then the equality case in $(\ref{inequality m})$ holds if and only if $\gamma_1,\ldots,\gamma_m$ are proportional in $H^{1,1}(M;\mathbb{R})$;

\item
if these $\gamma_1,\ldots,\gamma_n$ are nef and big classes then the equality case in $(\ref{inequality m})$ holds if and only if
the $(n-1,n-1)$ real cohomology classes $$\gamma_{i_1}\wedge\gamma_{i_2}\wedge\cdots\wedge
\gamma_{i_{m-1}}\wedge\gamma_{m+1}\wedge\cdots\wedge\gamma_n\qquad(1\leq i_1,\ldots,i_{m-1}\leq m)$$
 are all proportional in $H^{n-1,n-1}(M;\mathbb{R})$.
\end{enumerate}
\end{theorem}
\begin{proof}
The case $m=2$ is known due to Proposition \ref{AF kahler/nef} and Theorem \ref{equality case1}.

Assume that the assertions in Theorem \ref{appendix} hold for $m-1$. We shall show that  they must hold for $m$. Without loss of generality, we further assume that all the intersection numbers under consideration in the sequel are positive, which hold if the classes $\gamma_i$ are nef and big or K\"{a}hler. Then (\ref{inequality m}) can be obtained by approximating the nef classes $\gamma_i$ by K\"{a}hler classes.

Under the assumption condition we have
\be\label{13}(\gamma_1\cdots\gamma_m\gamma_{m+1}\cdots\gamma_n)^{m-1}\geq
\prod_{\mbox{\tiny$\begin{array}{c}
1\leq i\leq m\\
i\neq j\end{array}$}}(\gamma_i^{m-1}\gamma_j
\gamma_{m+1}\cdots\gamma_n),\qquad\forall~1\leq j\leq m.\ee

On the one hand, taking the product for $1\leq j\leq m$ in (\ref{13}) yields
\be\label{14}
\begin{split}
(\gamma_1\cdots\gamma_m\gamma_{m+1}\cdots\gamma_n)^{m(m-1)}
&\geq
\prod_{j=1}^m\prod_{\mbox{\tiny$\begin{array}{c}
1\leq i\leq m\\
i\neq j\end{array}$}}(\gamma_i^{m-1}\gamma_j
\gamma_{m+1}\cdots\gamma_n)\\
&=\prod_{\mbox{\tiny$\begin{array}{c}
1\leq i,j\leq m\\
i\neq j\end{array}$}}(\gamma_i^{m-1}\gamma_j
\gamma_{m+1}\cdots\gamma_n)\\
&=:T_m.\end{split}\ee

On the other hand, we have
\be\label{15}\begin{split}
(T_m)^{m-1}&=\prod_{\mbox{\tiny$\begin{array}{c}
1\leq i,j\leq m\\
i\neq j\end{array}$}}(\gamma_i^{m-1}\gamma_j
\gamma_{m+1}\cdots\gamma_n)^{m-1}\\
&\geq\prod_{\mbox{\tiny$\begin{array}{c}
1\leq i,j\leq m\\
i\neq j\end{array}$}}\big[(\gamma_i^m\gamma_{m+1}\cdots\gamma_n)^{m-2}
(\gamma_j^{m-1}\gamma_i\gamma_{m+1}\cdots\gamma_n)\big]~(\text{apply the assumption})\\
&=T_m\cdot\big[\prod_{i=1}^m
(\gamma_i^m\cdot\gamma_{m+1}\cdots\gamma_n)\big]^{(m-1)(m-2)}
\end{split}
\ee
and so
\be\label{16}T_m\geq\big[\prod_{i=1}^m
(\gamma_i^m\cdot\gamma_{m+1}\cdots\gamma_n)\big]^{m-1}.\ee
Combining (\ref{14}) with (\ref{16}) leads to the desired inequality (\ref{inequality m}). If further the equality case of (\ref{inequality m}) holds, then all the inequalities in (\ref{13}), (\ref{14}) and (\ref{15}) are indeed equalities. Applying the assumption of the equality case for $m-1$ easily deduce the expected equality case for $m$.
\end{proof}

\end{document}